\documentclass[10pt,a4paper]{scrartcl}%

\usepackage[T1]{fontenc}
\usepackage{epsfig}
\usepackage{graphicx}
\usepackage{amssymb}
\usepackage{amsmath}
\usepackage[utf8]{inputenc}
\usepackage{amsfonts}
\usepackage{amsthm}
\usepackage{dsfont}
\usepackage{comment}
\usepackage{enumerate}
\usepackage{mathbbol}
\usepackage{mathrsfs}

\usepackage[a4paper,vmargin={2cm,2cm},hmargin={2.25cm,2.25cm}]{geometry}

\usepackage{hyperref}
\usepackage[english]{babel}

\usepackage{listings}

\usepackage{stmaryrd}

\usepackage{pgf,tikz,pgfplots}

\pgfplotsset{compat=1.15}
\usetikzlibrary{arrows,shapes,positioning,calc}

\theoremstyle{plain}
\newtheorem{theorem}{Theorem}
\newtheorem{lemma}[theorem]{Lemma}

\newtheorem{corollary}[theorem]{Corollary}

\theoremstyle{definition}
\newtheorem{definition}[theorem]{Definition}

\theoremstyle{remark}

\newcommand{\esp}[1]{\ensuremath{\mathbb{E} \left[ #1 \right]}}
\newcommand{\proba}[1]{\ensuremath{\mathbb{P} \left( #1 \right)}}

\DeclareMathOperator{\spn}{span}

\date{2022}

\makeindex

\begin{document}

\title{On the independence number of random trees via tricolourations}
\author{Etienne Bellin \footnote{etienne.bellin@polytechnique.edu}
\\
{\normalsize CMAP - Ecole Polytechnique}
}

\maketitle

\begin{abstract}
We are interested in the independence number of large random simply generated trees and related parameters, such as their matching number or the kernel dimension of their adjacency matrix. We express these quantities using a canonical tricolouration, which is a way to colour the vertices of a tree with three colours. As an application we obtain limit theorems in $L^p$ for the renormalised independence number in large simply generated trees (including large size-conditioned Bienaymé-Galton-Watson trees).
\end{abstract}

\section{Introduction}

A subset $S$ of vertices of a finite graph $G$ is called an \textit{independent set} if there is no pair of connected vertices in $S$. The \textit{independence number} of $G$, denoted by $I(G)$, is the biggest cardinal of an independent set of $G$. The independence number is a well studied quantity in computational complexity theory. It is known that computing the independence number is NP-hard in general (see e.g. \cite[Sec.~3.1.3]{garey}). A lot of work has been carried out to describe algorithms computing the independence number in general graphs \cite{robson} \cite{xiao} and in special classes of graphs where the computational time can be decreased (e.g. cubic graphs \cite{xiaodeg3}, claw-free graphs \cite{minty}, $P_5$-free graphs \cite{lokshtanov}). The independence number has also received interest in combinatorics and in probability. Upper bounds have been found using probabilistic methods for cubic graphs \cite{bolobas}. Assymptotics have been found in certain classes of random trees (e.g. conditioned Bienaymé-Galton-Watson trees \cite{devroye}, simply generated trees \cite{banderier}, random recursive trees and binary search trees \cite{fuchs}, and a wider class of random trees constructed from a Crump–Mode–Jagers branching process \cite{janson2}). Finally we mention three articles giving applications of the independence number in scheduling theory \cite{joo}, coding theory \cite{butenko} and collusion detection in voting pools \cite{araujo}. 

The goal of this article is to study the independence number of large simply generated trees. Simply generated trees are a wide class of random plane trees (i.e.~rooted and ordered trees) introduced in \cite{meir} and encompass Bienaymé-Galton-Watson trees (BGW trees for short) conditioned to have a fixed number of vertices. Informally, a BGW tree with offspring distribution $\mu$ is a plane tree where vertices have an i.i.d.~number of children with law $\mu$. Various natural models of random trees are obtained with appropriate choices of the offspring distribution: e.g., uniform plane trees, uniform plane d-ary trees and uniform Cayley trees (see \cite[Sec.\,10]{janson}). In order to study the independence number, we will use a particular tricolouration of trees introduced in \cite{zito} and later studied in \cite{coulomb}, \cite{bauer} and \cite{chapoton}. This colouring is based on the notion of covering. A \textit{covering} of a finite tree $T$ is a subset of vertices $S$ of $T$ such that every edge of $T$ is adjacent to a vertex of $S$. A \textit{smallest covering} of $T$ is a covering with minimal cardinality. In general, a tree has more than one smallest covering. For every vertex $v$ of $T$ we colour $v$ in the following way:
\begin{itemize}
    \item If $v$ belongs to every smallest covering, we colour $v$ in green.
    \item If $v$ belongs to no smallest covering, we colour it in red.
    \item If $v$ belongs to some smallest coverings but not all, we colour it in orange.
\end{itemize}

\begin{figure}
\centering
\includegraphics[scale=0.4]{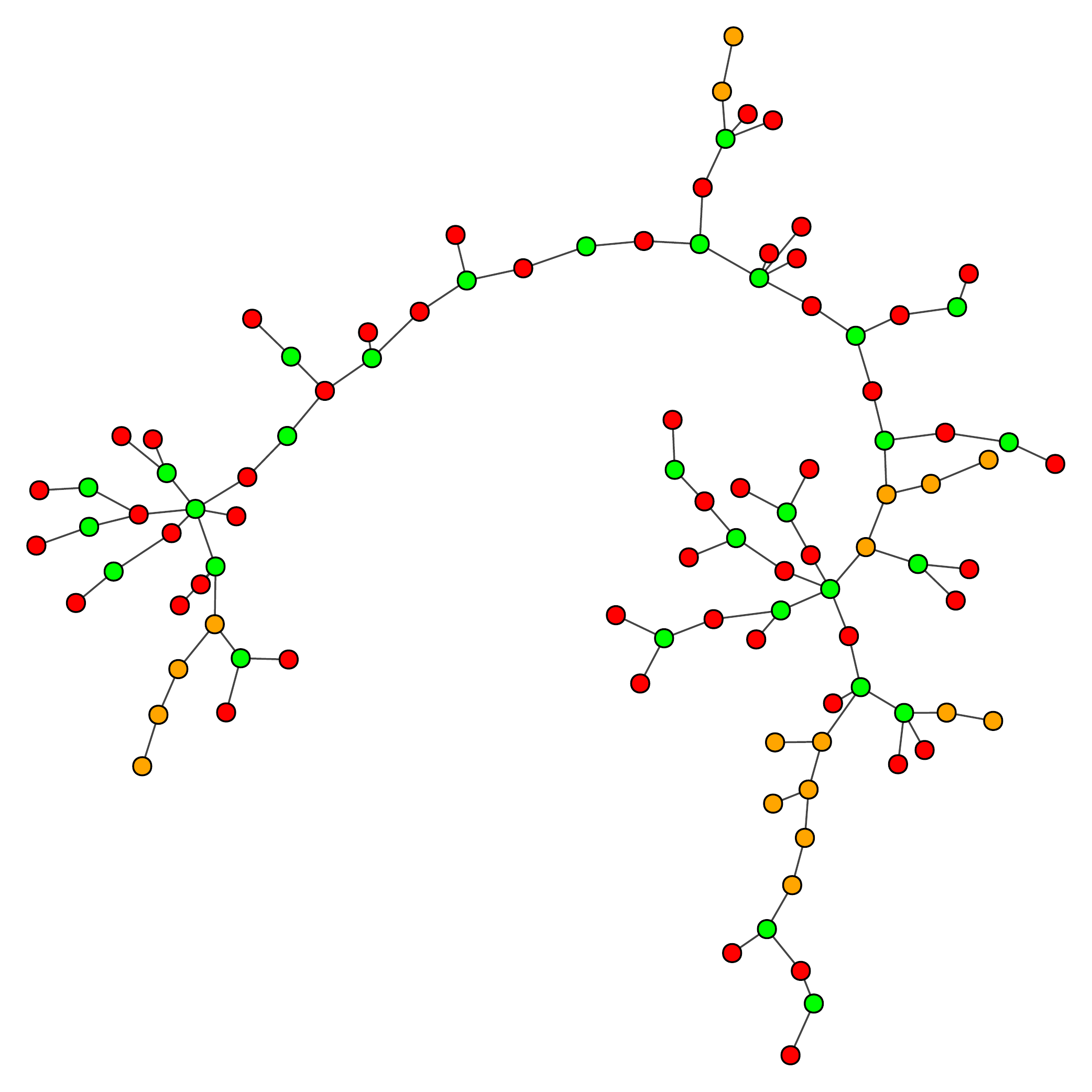}
\caption{Tricolouration of a BGW tree with 100 vertices and a Poisson offspring distribution of parameter 1. The algorithm used to tricolour this tree can be found in \cite[Appendix~A]{chapoton}.}
\end{figure}

\medskip

For a tree $T$, denote by $n_g(T)$, $n_o(T)$ and $n_r(T)$, respectively, the number of green, orange and red vertices in $T$. It has been noticed in \cite{bauer} that the size of a smallest covering of a tree $T$ is equal to $n_g(T)+n_o(T)/2$. Since the complementary of a smallest covering is an independent set of maximal size, the independence number of $T$ is $n_r(T)+n_o(T)/2$. Actually, other statistics of the tree $T$ can be expressed as a linear combination of $n_g(T)$, $n_o(T)$ and $n_r(T)$. For instance the \textit{matching number} $M(T)$ (i.e.~the maximum size of a partial vertex matching) is equal to the size of a smallest covering which is $n_g(T)+n_o(T)/2$. The \textit{nullity} $N(T)$ (i.e.~the kernel dimension of the adjacency matrix) is $n_r(T)-n_g(T)$. The \textit{edge cover number} and the \textit{clique cover number} also coincide with the independence number on trees (see \cite{fuchs}). Our main result (Theorem \ref{main thm}) concerns simply generated trees, but, to keep this introduction short, let us state here a particular case for critical BGW trees.
\begin{theorem}
\label{weak main thm}
Let $T_n$ be a Bienaymé-Galton-Watson tree with reproduction law $\mu$, conditionned on having $n$ vertices. Denote by $G(t):=\sum_{k=0}^\infty \mu_kt^k$ the generating function of $\mu$ and let $q$ be the unique solution of $G(1-q)=q$ in $[0,1]$. Suppose that $\mu$ has mean 1, then, the following convergences hold in $L^p$ for every $p>0$:
\begin{equation*}
\frac{n_g(T_n)}{n} \xrightarrow[n \to \infty]{L^p} \frac{1-q+(1-2q)G'(1-q)}{1+G'(1-q)}, ~~~~~~~~ \frac{n_o(T_n)}{n} \xrightarrow[n \to \infty]{L^p} \frac{2\,q\,G'(1-q)}{1+G'(1-q)},  
\end{equation*}
\begin{equation*}
    \frac{n_r(T_n)}{n} \xrightarrow[n \to \infty]{L^p} \frac{q}{1+G'(1-q)}.
\end{equation*}

\end{theorem}
Explicit computations of the expected number of green, orange and red vertices have been carried out in the case of a uniform Cayley tree with fixed size in \cite{bauer} using generating functions. This confirms the convergence, in mean value, of Theorem \ref{weak main thm}. Indeed, it is well known that a BGW tree with Poisson distribution of parameter 1 conditioned to have $n$ vertices has the same law as a uniform Cayley tree with $n$ vertices. As we said before, different quantities such as the independence number $I(T)$, the matching number $M(T)$ and the nullity $N(T)$ of a tree $T$ can be expressed in terms of $n_g(T)$, $n_o(T)$ and $n_r(T)$, so Theorem \ref{weak main thm} yields limit theorems for these in the case of critical BGW trees.
\begin{corollary}
\label{weak cor main thm}
With the same notation and hypothesis as in Theorem \ref{weak main thm}, the following convergences hold in $L^p$ for every $p>0$:
\begin{equation*}
\frac{I(T_n)}{n} \xrightarrow[n \to \infty]{L^p} q, ~~~~~~~~
\frac{M(T_n)}{n} \xrightarrow[n \to \infty]{L^p} 1-q, ~~~~~~~~
\frac{N(T_n)}{n} \xrightarrow[n \to \infty]{L^p} 2q-1.
\end{equation*}
\end{corollary}
With the same hypothesis, the authors of \cite{devroye} show the convergence of $I(T_n)/n$ in probability towards $q$. Moreover, in \cite{banderier} the convergence of the first and second moment of $I(T_n)/n$ is studied. More precisely, it is shown that $\esp{I(T_n)}=nq+o(1)$ and $V(I(T_n))=\nu n+o(1)$ for some constant $\nu$. 
The main tool to prove Theorem \ref{main thm} is the use of limit theorems for uniformly pointed simply generated trees found in \cite{stufler} (see Section \ref{sec limit thm}). In the first section we introduce simply generated trees. In the next section we state our main result in its most general form. To prove our main result, we first explain the limit theorems of \cite{stufler} in the third section. The next two sections give properties of the tricolouration and describe how to colour the limiting trees. Finally, in the last section we prove Theorem \ref{main thm}.

\section{Simply generated trees}
\label{sec simply gen trees}

Let $\mathbf w := (w_i)_{i \geq 0}$ be a sequence of nonnegative weights. A \textit{simply generated tree} having $n$ vertices with weight sequence $\mathbf w$ is a random plane tree $T_n$ such that for every finite plane tree $T$, 
\begin{equation*}
\proba{T_n=T}=\frac{1}{Z_n} \left(\prod_{v \in T} w_{k_v}\right) \mathds{1}_{|T|=n}
\end{equation*}
where $k_v$ is the outdegree of the vertex $v$ in $T$, $|T|$ is the number of vertices in $T$ and $Z_n$ is the normalising constant defined by
\begin{equation*}
Z_n := \sum_{|T|=n} \prod_{v \in T} w_{k_v}.
\end{equation*}
Notice that, when the weight sequence $\mathbf{w}$ is actually a probability sequence (i.e.~the sum of the weights is equal to 1) then we recover the class of BGW trees. For $T_n$ to be well defined, one needs $Z_n$ to be nonzero. First of all, suppose that $w_0 > 0$ and $w_k > 0$ for some $k \geq 1$ otherwise $Z_n =0$ for all $n \geq 1$. Let $\spn(\mathbf w) := \gcd \{i\geq 0\,|\, w_i >0\}$ (since $w_k>0$ this quantity is well defined). The following result, found for instance in \cite[Cor.~15.6]{janson}, characterises the $n$'s such that $Z_n > 0$:
\begin{lemma}[Janson 2012]
If $Z_n > 0$ then $n \equiv 1 \mod \spn(\mathbf w)$. Conversely, there exists $n_0$ such that for all $n \geq n_0$ satisfying $n \equiv 1 \mod \spn(\mathbf w)$, $Z_n >0$.
\end{lemma}

\paragraph*{Throughout this document and in Theorem \ref{main thm}, we suppose that $w_0 > 0$, $w_k > 0$ for some $k \geq 1$ and that all the $n$'s appearing satisfy $n \geq n_0$ and $n \equiv 1 \mod \spn(\mathbf w)$.}
Let $\rho \in [0,+\infty]$ be the radius of convergence of the generating series
\begin{equation*}
\phi(x) := \sum_{i \geq 0} w_i x^i.
\end{equation*}
It is shown in \cite[Lemma 3.1]{janson} that, if $\rho > 0$, then the function defined by
\begin{equation*}
\psi(x) := \frac{x \phi'(x)}{\phi(x)}
\end{equation*}
is increasing on $[0,\rho)$ and we can define $\nu:=\lim_{x \to \rho} \psi(x) \in (0,+\infty]$. We distinguish three different regimes:
\begin{itemize}
    \item Regime 1 when $\rho >0$ and $\nu \geq 1$. In this case there is a unique $\tau \in [0,\rho]$ such that $\tau< +\infty$ and $\psi(\tau)=1$.
    \item Regime 2 when $\rho >0$ and $0<\nu<1$. In this case $\rho < +\infty$ and we set $\tau := \rho$.
    \item Regime 3 when $\rho =0$.
\end{itemize}
In regime 1 and 2 we can define a probability function given by
\begin{equation}
\label{eq proba equivalente}
\pi_k := \frac{\tau^k w_k}{\phi(\tau)}.    
\end{equation}
The associated mean and generating function are respectively given by
\begin{equation}
\label{eq mean and gen function}
m := \min(1,\nu) ~~~~\text{ and }~~~~G(x):=\frac{\phi(\tau x)}{\phi(x)}.
\end{equation}
An important result of \cite{janson} is that, $T_n$, in regime 1 or 2, has the same law as a BGW tree with reproduction law $\pi$ conditionned to have $n$ vertices. In regime 3, $T_n$ is not distributed like a conditionned BGW tree. Note that a critical or super-critical BGW tree or a BGW tree with a reproduction law with infinite mean, conditioned to have $n$ vertices, always lays in regime 1. Moreover for a critical BGW tree with reproduction law $\mu$, the probability $\pi$ is the same as $\mu$. A sub-critical BGW tree, conditionned to have $n$ vertices, is either in regime 1 or in regime 2. We define $\textit{complete condensation}$ to be the condition:
\begin{equation}
\label{eq condensation}
\Delta(T_n) = (1-m)n+n\mathcal E_n
\end{equation}
where $\Delta(T_n)$ is the maximum degree of a vertex of $T_n$ and $\mathcal E_n$ is a random variable converging in probability towards 0. For instance, complete condensation happens in regime 2 when there exists $\theta >1$ and a slowly varying function $\ell$ such that $\pi_k = \ell(k)k^{-(1+\theta)}$ (see \cite{kortchemski}). Complete condensation also happens in regime 3 for the weight sequence $w_k = k!^\alpha$ for $\alpha>0$ (see \cite[Ex.~19.36]{janson}).

\section{Main results}
\label{sec main results}

In this section we state our main results. We keep all the notations and assumptions of Section \ref{sec simply gen trees}. 

\begin{theorem}
\label{main thm}
Let $T_n$ be a simply generated tree with $n$ vertices according to the weight sequence $\mathbf w = (w_i)_{i\geq 0}$. Recall that, in regime 1 and 2, $G$ denotes the generating function of $\pi$ defined in (\ref{eq proba equivalente}). Let $q$ be the unique solution of $q=G(1-q)$ in $[0,1]$.
\begin{enumerate}
\item In regime 1 and in regime 2 with complete condensation (meaning that (\ref{eq condensation}) is satisfied), the following convergences hold in $L^p$ for every $p>0$:
\begin{equation*}
\frac{n_g(T_n)}{n} \xrightarrow[n \to \infty]{L^p} \frac{1-q+(1-2q)G'(1-q)}{1+G'(1-q)}, ~~~~~~~~ \frac{n_o(T_n)}{n} \xrightarrow[n \to \infty]{L^p} \frac{2\,q\,G'(1-q)}{1+G'(1-q)},  
\end{equation*}
\begin{equation*}
\frac{n_r(T_n)}{n} \xrightarrow[n \to \infty]{L^p} \frac{q}{1+G'(1-q)}.
\end{equation*}

\item In regime 3 with complete condensation, the following convergences hold in $L^p$ for every $p>0$:
\begin{equation*}
\frac{n_g(T_n)}{n} \xrightarrow[n \to \infty]{L^p} 0, ~~~~~~~~ \frac{n_o(T_n)}{n} \xrightarrow[n \to \infty]{L^p} 0, ~~~~~~~~ \frac{n_r(T_n)}{n} \xrightarrow[n \to \infty]{L^p} 1.
\end{equation*}

\end{enumerate}
\end{theorem}

\begin{corollary}
\label{cor main thm}
We keep the same notation and hypothesis as in Theorem \ref{main thm}. Recall that $I(T_n)$, $M(T_n)$ and $N(T_n)$ are, respectively, the independence number, the matching number and the nullity of $T_n$. 
\begin{enumerate}
\item In regime 1 and in regime 2 with complete condensation, the following convergences hold in $L^p$ for every $p>0$:
\begin{equation*}
\frac{I(T_n)}{n} \xrightarrow[n \to \infty]{L^p} q, ~~~~~~~~
\frac{M(T_n)}{n} \xrightarrow[n \to \infty]{L^p} 1-q, ~~~~~~~~
\frac{N(T_n)}{n} \xrightarrow[n \to \infty]{L^p} 2q-1.
\end{equation*}
\item In regime 3 with complete condensation, the following convergences hold in $L^p$ for every $p>0$:
\begin{equation*}
\frac{I(T_n)}{n} \xrightarrow[n \to \infty]{L^p} 1, ~~~~~~~~
\frac{M(T_n)}{n} \xrightarrow[n \to \infty]{L^p} 0, ~~~~~~~~
\frac{N(T_n)}{n} \xrightarrow[n \to \infty]{L^p} 1.
\end{equation*}
\end{enumerate}
\end{corollary}

\section{Limit theorems for uniformly pointed simply generated trees}
\label{sec limit thm}

In this section we explain the results proved in \cite{stufler} which will be our basic tool to prove Theorem \ref{main thm}. All the proofs and details of this section can be found in the above mentioned article. As said in the introduction, these results are limit theorems for uniformly pointed simply generated trees. A \textit{pointed tree} is simply a couple $(T,v)$ with a plane tree (i.e.~rooted and ordered tree) $T$ and a distinguished vertex $v$ of $T$. A uniformly pointed simply generated tree is a couple $(T_n,v_n)$ where $T_n$ is a simply generated tree with $n$ vertices and $v_n$ is a distinguished vertex chosen uniformly at random among the $n$ vertices of $T_n$. Basically, in regime 1 and in regime 2 and 3 with complete condensation, the local tree structure around $v_n$ converges towards an infinite random tree which depends only on the regime. To formally define the notion of convergence used here, one needs to consider $T_n$ to be a subtree of a big ambient tree denoted by $\mathcal U_\infty^\bullet$. Every plane tree (e.g. $T_n$) is considered, by definition, to be rooted, however we will encounter some infinite trees without any root which is unusual in the classic framework of plane trees. Let
\begin{equation*}
\mathcal V_\infty := \left\{\emptyset\right\} \cup \bigcup_{n \geq 1} (\mathbb N^*)^n 
\end{equation*}
be the set of words (empty word included) formed in the alphabet $\mathbb N^* = \{1,2,3,\dots\}$. Usually, plane trees are defined as subtrees of the so-called \textit{Ulam-Haris tree}, denoted here by $\mathcal U_\infty$, which is the tree with vertex set $\mathcal V_\infty$ and edge set $\left\{ (a_1 \dots a_{n-1};a_1 \dots a_n) \,|\, \forall n,a_1,\dots,a_n\in \mathbb N^*\right\}$. With this definition, all the plane trees have a root which is a common ancestor to every vertex of the tree (it is the vertex designated by the empty word $\emptyset$). However, in regime 1, the root of $T_n$ is, in a local point of view, at infinite distance from the distinguished vertex $v_n$. It suggests that the local limit of $T_n$ around $v_n$ has an infinite spine of ancestors and therefore, has no root. This is why we need a more general framework than the usual one for plane trees. Here we explain informally the construction of $\mathcal U_\infty^\bullet$. Let $u_0,u_1,\dots$ be vertices, in the plane, lined up to form an infinite connected spine. Each vertex $u_i$ with $i>0$ gets an infinite countable number of children on the left and on the right of its child $u_{i-1}$. Then, all the leaves of the current tree ($u_0$ included) give birth to the Ulam-Harris tree $\mathcal U_\infty$. The tree we obtain from this construction is denoted by $\mathcal U_\infty^\bullet$ and its set of vertices is denoted by $\mathcal V_\infty^\bullet$ (see Figure \ref{fig arbre ulam harris general}). To formally define this tree one could start by creating the vertex set $\mathcal V_\infty^\bullet$ as a subset of $\mathbb N \times \mathbb Z \times \mathcal V_\infty$ and then describing the edge set. However we think that the above informal construction is enough for our purpose. As we said, we want to represent $T_n$ as a subtree of $\mathcal U_\infty^\bullet$. First we make clear what we call a subtree of $\mathcal U_\infty$ and $\mathcal U_\infty^\bullet$. Denote by $\mathcal E_\infty$ the edge set of $\mathcal U_\infty$ and $\mathcal E_\infty^\bullet$ the edge set of $\mathcal U_\infty^\bullet$.
\begin{figure}
\begin{center}
    \begin{tikzpicture}
    [vertexdot/.style = {draw,circle,fill,inner sep=1.5pt}]
    \node[vertexdot,label=$\emptyset$] (0) at (0,0) {};
    \node[vertexdot,label={[below left]1}] (1) at (-0.8,-1) {};
    \node[vertexdot,label={[below right]2}] (2) at (0,-1) {};
    \node[vertexdot,label={[below right]3}] (3) at (0.8,-1) {};
    \node[vertexdot,label={[below left]31}] (31) at (0.5,-2) {};
    \node[vertexdot,label={[below right]32}] (32) at (1.1,-2) {};
    \node[vertexdot,label={[below left]11}] (11) at (-0.8,-2) {};
    \node[vertexdot,label={[below left]321}] (321) at (1.1,-3) {};
    
    \draw (0)--(1)--(11);
    \draw (0)--(2);
    \draw (0)--(3)--(31);
    \draw (3)--(32)--(321);
    \end{tikzpicture}
    \hspace{2em}
    \begin{tikzpicture}
    [vertexdot/.style = {draw,circle,fill,inner sep=1.5pt}]
    \node[vertexdot,label={[right]$u_0$}] (u0) at (0,0) {};
    \node[vertexdot,label={[right]$u_1$}] (u1) at (0,1.3) {};
    \node[vertexdot,label={[right]$u_2$}] (u2) at (0,2.6) {};
    \node[vertexdot] (u1a) at (-1,0) {};
    \node[vertexdot] (u1b) at (-2,0) {};
    \node[vertexdot] (u1c) at (1,0) {};
    \node[vertexdot] (u1d) at (2,0) {};
    \node[vertexdot] (u2a) at (-1,1.3) {};
    \node[vertexdot] (u2b) at (-2,1.3) {};
    \node[vertexdot] (u2c) at (1,1.3) {};
    \node[vertexdot] (u2d) at (2,1.3) {};
    
    \draw (u0)--(u1)--(u2);
    \draw (0,2.6)--(0,2.8);
    \draw[dashed] (0,2.8)--(0,3.3);
    \draw[dashed] (2.5,1.3)--(3,1.3);
    \draw[dashed] (2.5,0)--(3,0);
    \draw[dashed] (-2.5,1.3)--(-3,1.3);
    \draw[dashed] (-2.5,0)--(-3,0);
    \draw (u1a)--(u1)--(u1b);
    \draw (u1c)--(u1)--(u1d);
    \draw (u2a)--(u2)--(u2b);
    \draw (u2c)--(u2)--(u2d);
    
    \draw (u0) to[out=-45,in=225,distance=1cm] (u0);
    \draw (u1a) to[out=-45,in=225,distance=1cm] (u1a);
    \draw (u1b) to[out=-45,in=225,distance=1cm] (u1b);
    \draw (u1c) to[out=-45,in=225,distance=1cm] (u1c);
    \draw (u1d) to[out=-45,in=225,distance=1cm] (u1d);
    \draw (u2a) to[out=-45,in=225,distance=1cm] (u2a);
    \draw (u2b) to[out=-45,in=225,distance=1cm] (u2b);
    \draw (u2c) to[out=-45,in=225,distance=1cm] (u2c);
    \draw (u2d) to[out=-45,in=225,distance=1cm] (u2d);
    \end{tikzpicture}
\caption{On the left, a subtree of the Ulam-Harris tree $\mathcal U_\infty$. On the right, the tree $\mathcal U_\infty^\bullet$. Each loop represents a copy of the Ulam-Harris tree.}
\label{fig arbre ulam harris general}
\end{center}
\end{figure}
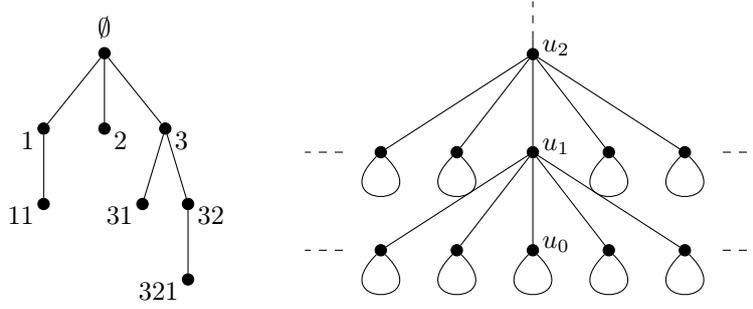
\begin{definition}
\label{defi sous arbre}
A \textit{subtree} $t$ of $\mathcal U_\infty$ is a tree with vertex set included in $\mathcal V_\infty$ and edge set included in $\mathcal E_\infty$, such that the vertex $\emptyset$ belongs to $t$ and such that there is no holes in $t$, meaning that: if $v=(a_1\dots a_n)$ is a vertex of $t$ with $a_n>0$ then $(a_1\dots a_n-1)$ is also a vertex of $t$. Similarly a \textit{subtree} $t$ of $\mathcal U_\infty^\bullet$ is a tree with vertex set included in $\mathcal V_\infty^\bullet$ and edge set included in $\mathcal E_\infty^\bullet$, such that the vertex $u_0$ belongs to $t$ and such that there is no holes in $t$ (see Figure \ref{fig presque fringe subtree}). A subtree of $\mathcal U_\infty^\bullet$ is \textit{rooted} if the set $\{k\geq 0 \,\vert\, u_k \in t\}$ is finite. In this case the vertex $u_k$ with maximal $k$ in $t$ is called the \textit{root} of $t$.
\end{definition}
The representation of $T_n$ as a subtree of $\mathcal U_\infty^\bullet$ will obviously depend on the distinguished vertex $v_n$ since we want to look at the local structure around this vertex. More precisely, let $T$ be a plane tree and $v$ be a vertex of $T$. We identify the distinguished vertex $v$ with the element $u_0$ and the root of $T$ is identified with the element $u_h$ where $h$ is the graph distance between the root and $v$ in $T$ (it is the height of $v$). All the other vertices of $T$ are identified such that the plane order is preserved (See Figure \ref{fig representation arbre plan}). We denote by $(T,v)$ the subtree of $\mathcal U_\infty^\bullet$ obtained this way. We can now formally define the notion of convergence we use.
\begin{figure}
\begin{center}
\begin{tikzpicture}
    [vertexdot/.style = {draw,circle,fill,inner sep=1.5pt}]
    \node[vertexdot,label=root] (0) at (0,0) {};
    \node[vertexdot] (1) at (-0.8,-1.3) {};
    \node[vertexdot] (2) at (0,-1.5) {};
    \node[vertexdot] (3) at (0.8,-1.3) {};
    \node[vertexdot] (21) at (-0.5,-2.5) {};
    \node[vertexdot,label={[right]$v$}] (22) at (0.5,-2.5) {};
    
    \node (a) at (-0.8,-1.7) {$t_1$};
    \node (b) at (0.8,-1.7) {$t_2$};
    \node (c) at (-0.5,-2.9) {$t_3$};
    \node (d) at (0.5,-2.9) {$t_4$};
    
    \draw (0)--(1);
    \draw (0)--(2)--(21);
    \draw (2)--(22);
    \draw (0)--(3);
    
    \draw (1) to[out=-45,in=225,distance=1cm] (1);
    \draw (3) to[out=-45,in=225,distance=1cm] (3);
    \draw (21) to[out=-45,in=225,distance=1cm] (21);
    \draw (22) to[out=-45,in=225,distance=1cm] (22);
    \end{tikzpicture}
    \hspace{2em}
    \begin{tikzpicture}
    [vertexdot/.style = {draw,circle,fill,inner sep=1.5pt},
    vertexcirc/.style = {draw,circle,inner sep=1.5pt}]
    \node[vertexdot,label={[right]$u_0$}] (u0) at (0,0) {};
    \node[vertexdot,label={[right]$u_1$}] (u1) at (0,1.3) {};
    \node[vertexdot,label={[right]$u_2$}] (u2) at (0,2.6) {};
    \node[vertexdot] (u1a) at (-1,0) {};
    \node[vertexcirc,opacity=0.5] (u1b) at (-2,0) {};
    \node[vertexcirc,opacity=0.5] (u1c) at (1,0) {};
    \node[vertexcirc,opacity=0.5] (u1d) at (2,0) {};
    \node[vertexdot] (u2a) at (-1,1.3) {};
    \node[vertexcirc,opacity=0.5] (u2b) at (-2,1.3) {};
    \node[vertexdot] (u2c) at (1,1.3) {};
    \node[vertexcirc,opacity=0.5] (u2d) at (2,1.3) {};
    
    \node (a) at (-1,0.9) {$t_1$};
    \node (b) at (1,0.9) {$t_2$};
    \node (c) at (-1,-0.4) {$t_3$};
    \node (d) at (0,-0.4) {$t_4$};

    \draw[line width=0.5mm] (u0)--(u1)--(u2);
    \draw[opacity=0.5] (0,2.6)--(0,2.8);
    \draw[dashed,opacity=0.5] (0,2.8)--(0,3.3);
    \draw[dashed,opacity=0.5] (2.5,1.3)--(3,1.3);
    \draw[dashed,opacity=0.5] (2.5,0)--(3,0);
    \draw[dashed,opacity=0.5] (-2.5,1.3)--(-3,1.3);
    \draw[dashed,opacity=0.5] (-2.5,0)--(-3,0);
    \draw[line width=0.5mm] (u1a)--(u1);
    \draw[opacity=0.5] (u1b)--(u1);
    \draw[opacity=0.5] (u1c)--(u1)--(u1d);
    \draw[line width=0.5mm] (u2a)--(u2);
    \draw[opacity=0.5] (u2b)--(u2);
    \draw[line width=0.5mm] (u2c)--(u2);
    \draw[opacity=0.5] (u2d)--(u2);
    
    \draw[line width=0.5mm] (u0) to[out=-45,in=225,distance=1cm] (u0);
    \draw[line width=0.5mm] (u1a) to[out=-45,in=225,distance=1cm] (u1a);
    \draw[opacity=0.5] (u1b) to[out=-45,in=225,distance=1cm] (u1b);
    \draw[opacity=0.5] (u1c) to[out=-45,in=225,distance=1cm] (u1c);
    \draw[opacity=0.5] (u1d) to[out=-45,in=225,distance=1cm] (u1d);
    \draw[line width=0.5mm] (u2a) to[out=-45,in=225,distance=1cm] (u2a);
    \draw[opacity=0.5] (u2b) to[out=-45,in=225,distance=1cm] (u2b);
    \draw[line width=0.5mm] (u2c) to[out=-45,in=225,distance=1cm] (u2c);
    \draw[opacity=0.5] (u2d) to[out=-45,in=225,distance=1cm] (u2d);
    \end{tikzpicture}
\caption{On the left, a finite tree $T$ with a distinguished vertex $v$ at distance 2 from the root. The loops $t_1,t_2,t_3$ and $t_4$ represent subtrees of $T$ which can be seen as subtrees of the Ulam-Harris tree $\mathcal U_\infty$. On the right, the representation of $T$ as a subtree of $\mathcal U_\infty^\bullet$.}
\label{fig representation arbre plan}   
\end{center}
\end{figure}
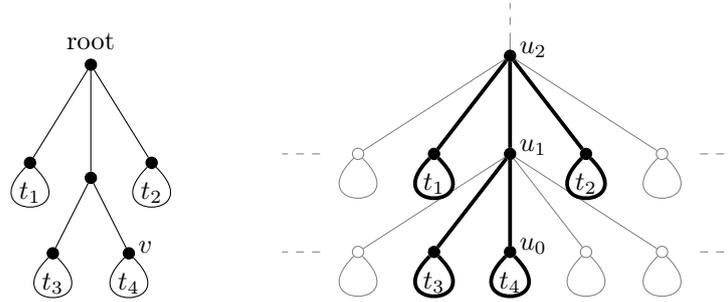
\begin{definition}
\label{defi conv}
Let $t_n$ and $t$ be subtrees of $\mathcal U_\infty^\bullet$ for all $n$. We say that $(t_n)$ converges towards $t$ and write $t_n \to t$ if for all $v \in \mathcal V_\infty^\bullet$,
\begin{equation*}
\mathds{1}_{v \in t_n} \xrightarrow[n \to \infty]{} \mathds{1}_{v \in t}.
\end{equation*}
\end{definition}
This notion of convergence induces a topology that is metrizable and compact over the set of subtrees of $\mathcal U^\bullet_\infty$. Before stating the limit theorems, one needs to define the limiting trees $T_1^*$, $T_2^*$ and $T_3^*$ (seen as subtrees of $\mathcal U^\bullet_\infty$) that correspond, respectively, to regime 1, 2 and 3. Let $T$ be a BGW tree with reproduction law $\pi$, given by (\ref{eq proba equivalente}), in regime 1 or 2. Let $\hat{\pi}$ be the probability measure on $\mathbb N \cup \{\infty\}$ given by
\begin{equation*}
\hat{\pi}_k = k\pi_k ~~ \forall k\in \mathbb N ~~\text{ and }~~ \hat{\pi}_\infty=1-m
\end{equation*}
where $m :=\min(1,\nu)$ is the mean of $\pi$. 

\begin{itemize}
    \item First we define $T_1^*$ in regime 1. Notice that in this case $\hat \pi_\infty=0$. We attach to $u_0$ an independent copy of $T$. For $k \geq 1$, $u_k$ receives offspring according to an independent copy of $\hat \pi$. Then $u_{k-1}$ is identified with a child of $u_k$ chosen uniformly at random. Finally, we attach an independent copy of $T$ to all the children of $u_k$, except $u_{k-1}$ (see Figure \ref{fig arbres limites}). 
    \item Now we define $T_2^*$ in regime 2. In this case $\hat \pi _\infty >0$. We attach to $u_0$ an independent copy of $T$. For $k \geq 1$, $u_k$ receives offspring according to an independent copy of $\hat \pi$. Notice that almost surely, there exist $1 \leq i < j$ two integers such that $u_1,\dots,u_{i-1},u_{i+1},\dots,u_{j-1}$ have a finite number of children and $u_i$ and $u_j$ have an infinite number of children. For every $k \in \{1,\dots,i-1,i+1,\dots,j-1\}$, $u_{k-1}$ is identified with a child of $u_k$ chosen uniformly at random, while $u_i$ gets infinitely many children on the left and the right of its child $u_{i-1}$. Finally, for all $k \geq 1$, we attach an independent copy of $T$ to all the children of $u_k$, except $u_{k-1}$. The tree $T_2^*$ is the tree obtained by keeping all the descendants of $u_{j-1}$ (see Figure \ref{fig arbres limites}). 
    \item Finally, $T_3^*$ is simply composed of the vertex $u_1$ having infinitely many children on the left and on the right of $u_0$, all of them, including $u_0$, being leaves (see Figure \ref{fig arbres limites}).
\end{itemize}

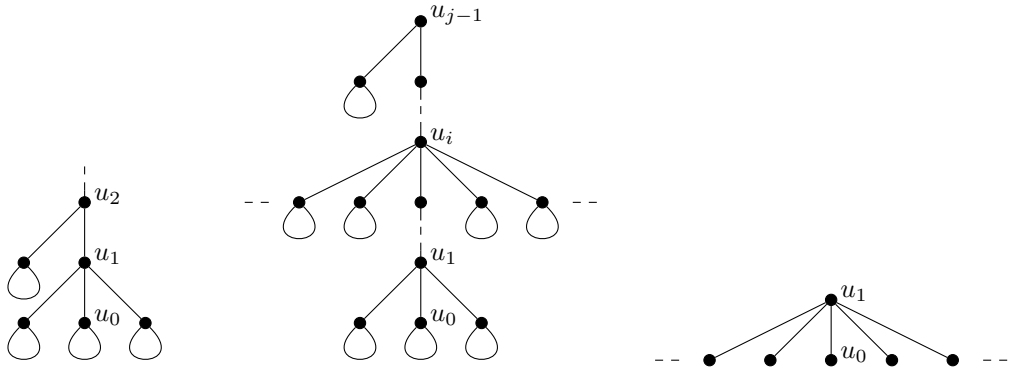
\begin{figure}
\begin{center}
    \begin{tikzpicture}
    [vertexdot/.style = {draw,circle,fill,inner sep=1.5pt},
    scale=0.8]
    \node[vertexdot,label={[right]$u_0$}] (u0) at (0,0) {};
    \node[vertexdot,label={[right]$u_1$}] (u1) at (0,1) {};
    \node[vertexdot,label={[right]$u_2$}] (u2) at (0,2) {};
    \node[vertexdot] (u1a) at (-1,0) {};
    \node[vertexdot] (u1c) at (1,0) {};
    \node[vertexdot] (u2a) at (-1,1) {};
    
    \draw (u0)--(u1)--(u2);
    \draw (0,2)--(0,2.2);
    \draw[dashed] (0,2.2)--(0,2.7);
    \draw (u1a)--(u1);
    \draw (u1c)--(u1);
    \draw (u2a)--(u2);
    
    \draw (u0) to[out=-45,in=225,distance=1cm] (u0);
    \draw (u1a) to[out=-45,in=225,distance=1cm] (u1a);
    \draw (u1c) to[out=-45,in=225,distance=1cm] (u1c);
    \draw (u2a) to[out=-45,in=225,distance=1cm] (u2a);
    \end{tikzpicture}
    \hspace{1em}
    \begin{tikzpicture}
    [vertexdot/.style = {draw,circle,fill,inner sep=1.5pt},
    scale=0.8]
    \node[vertexdot,label={[right]$u_0$}] (u0) at (0,0) {};
    \node[vertexdot,label={[right]$u_1$}] (u1) at (0,1) {};
    \node[vertexdot,label={[right]$u_i$}] (ui) at (0,3) {};
    \node[vertexdot] (ui-1) at (0,2) {};
    \node[vertexdot] (uj-2) at (0,4) {};
    \node[vertexdot,label={[right]$u_{j-1}$}] (uj-1) at (0,5) {};
    \node[vertexdot] (u1a) at (-1,0) {};
    \node[vertexdot] (u1c) at (1,0) {};
    \node[vertexdot] (uia) at (-1,2) {};
    \node[vertexdot] (uib) at (-2,2) {};
    \node[vertexdot] (uic) at (1,2) {};
    \node[vertexdot] (uid) at (2,2) {};
    \node[vertexdot] (uj-1a) at (-1,4) {};
    
    \draw (u0)--(u1);
    \draw (0,1)--(0,1.2);
    \draw[dashed] (0,1.2)--(0,1.7);
    \draw (0,1.7)--(0,2);
    \draw (0,3)--(0,3.2);
    \draw[dashed] (0,3.2)--(0,3.7);
    \draw (0,3.7)--(0,4);
    \draw[dashed] (2.5,2)--(3,2);
    \draw[dashed] (-2.5,2)--(-3,2);
    \draw (u1a)--(u1);
    \draw (u1c)--(u1);
    \draw (uia)--(ui)--(uib);
    \draw (uic)--(ui)--(uid);
    \draw (uj-1)--(uj-2);
    \draw (ui-1)--(ui);
    \draw (uj-1)--(uj-1a);
    
    \draw (u0) to[out=-45,in=225,distance=1cm] (u0);
    \draw (u1a) to[out=-45,in=225,distance=1cm] (u1a);
    \draw (u1c) to[out=-45,in=225,distance=1cm] (u1c);
    \draw (uia) to[out=-45,in=225,distance=1cm] (uia);
    \draw (uib) to[out=-45,in=225,distance=1cm] (uib);
    \draw (uic) to[out=-45,in=225,distance=1cm] (uic);
    \draw (uid) to[out=-45,in=225,distance=1cm] (uid);
    \draw (uj-1a) to[out=-45,in=225,distance=1cm] (uj-1a);
    \end{tikzpicture}
    \hspace{1em}
    \begin{tikzpicture}
    [vertexdot/.style = {draw,circle,fill,inner sep=1.5pt},
    scale=0.8]
    \node[vertexdot,label={[right]$u_1$}] (u1) at (0,1) {};
    \node[vertexdot,label={[right]$u_0$}] (u0) at (0,0) {};
    \node[vertexdot] (u1a) at (-1,0) {};
    \node[vertexdot] (u1b) at (-2,0) {};
    \node[vertexdot] (u1c) at (1,0) {};
    \node[vertexdot] (u1d) at (2,0) {};
    
    \draw (u1a)--(u1)--(u1b);
    \draw (u1c)--(u1)--(u1d);
    \draw (u0)--(u1);
    \draw[dashed] (2.5,0)--(3,0);
    \draw[dashed] (-2.5,0)--(-3,0);
    
    \end{tikzpicture}
\caption{Representation of the trees $T_1^*$, $T_2^*$ and $T_3^*$, respectively, from left to right. Each loop represents a copy of a Bienaymé-Galton-Watson tree of reproduction law $\pi$.}
\label{fig arbres limites}
\end{center}
\end{figure}

\begin{theorem}[Stufler 2018]
\label{limit thm}
Let $(T_n,v_n)$ be a uniformly pointed simply generated tree with $n$ vertices. Suppose that we are in regime $i=1$ or in regime $i \in \{2,3\}$ with complete condensation (meaning that (\ref{eq condensation}) is satisfied). Then the convergence
\begin{equation*}
(T_n,v_n) \xrightarrow[n\to \infty]{(d)} T_i^*
\end{equation*}
holds in distribution for the topology induces by the convergence of Definition \ref{defi conv}. 
\end{theorem}

\section{Properties of the tricolouration}
\label{sec prop tricolour}

In this section, we look at some general properties of the tricolouration defined in the introduction that will be essential to prove our main result. Let $T_1,\dots,T_n$ be $n$ rooted finite trees. We define $T_1*\cdots *T_n$ the rooted tree obtained by creating an edge between each root of $T_1,\dots,T_n$ and a new vertex which will be the root of $T_1*\cdots *T_n$. In particular the number of vertices $\# V(T_1*\cdots * T_n)$ equals $1+\# V(T_1) + \cdots + \# V(T_n)$. And the number of edges $\# E(T_1*\cdots * T_n)$ equals $n+\# E(T_1) + \cdots + \# E(T_n)$. We say that a rooted tree has colour $c$ if the root has colour $c$.

\begin{lemma}
\label{lemme jointures}
Let $T_1,\dots,T_n$ be $n$ rooted finite trees. Set $T := T_1*\cdots *T_n$, then
\begin{enumerate}
    \item $T$ is red if $T_1,\dots,T_n$ are all non-red.
    \item $T$ is orange if exactly one tree among $T_1,\dots,T_n$ is red.
    \item $T$ is green if two or more trees among $T_1,\dots,T_n$ are red.
\end{enumerate}
\end{lemma}

\begin{proof}
Denote by $C(T)$ the size of a smallest covering of $T$. Notice that $C(T)$ is equal to $C(T_1)+\cdots+C(T_n) + \Delta$ with $\Delta \in \{0,1\}$. More precisely, $C(T)=C(T_1)+\cdots+C(T_n)$ if and only if all the $T_i$'s are non-red. 
\begin{enumerate}
    \item Suppose that $T_1,\dots,T_n$ are all non-red. We can take a smallest covering for each $T_i$ such that the root of $T_i$ is included in the covering. Then the union of these coverings gives a smallest covering of $T$. Moreover we can see that all the smallest coverings of $T$ are obtained this way. Thus $T$ is red.
    \item Suppose than $T_1$ is red and $T_2,\dots,T_n$ are all non-red. We can take a smallest covering for each $T_i$ in addition to the root of $T$. This gives a smallest covering of $T$, so $T$ is either green or orange. We can also take a smallest covering for each $T_i$, $i>1$, such that the root of $T_i$ is included in the covering, a smallest covering of $T_1$ and the root of $T_1$. This also gives a smallest covering of $T$. Thus $T$ is orange.
    \item Suppose that $T_1$ and $T_2$ are red. As for the previous case, we can take a smallest covering for each $T_i$ in addition to the root of $T$. This gives a smallest covering for $T$. But, as opposed to the previous case, all the smallest coverings of $T$ are obtained this way. Thus $T$ is green.
\end{enumerate} 
\end{proof}

If $T_1,T_2$ are finite trees and $v_1,v_2$ are vertices of, respectively, $T_1$ and $T_2$, then we denote by $(T_1,v_1)*(T_2,v_2)$ the tree obtained from $T_1$ and $T_2$ by drawing an edge between $v_1$ and $v_2$.

\begin{lemma}
\label{lemme jointure}
With the same notation as above, set $T := (T_1,v_1)*(T_2,v_2)$. If $v_1$ is green in the tricolouration of $T_1$ then, the colour of every vertex $v$ in the tricouloration of $T$ is just the same as its colour in the tricolouration of $T_1$ (if $v$ is a vertex of $T_1$) or $T_2$ (if $v$ is a vertex of $T_2$). In other words, the tricoloured tree $T$ is simply obtained by drawing an edge between $v_1$ and $v_2$ and keeping the colours of $T_1$ and $T_2$.
\end{lemma}

\begin{proof}
Notice that a smallest covering of $T_1$ combined with a smallest covering a $T_2$ gives a smallest covering of $T$. Conversely a smallest covering of $T$ is necessarily obtained by combining a smallest covering of $T_1$ and $T_2$.
\end{proof}

\section{Tricolouration of the infinite limiting trees}
\label{sec tricoloration arbre infini}

In this section, we extend the definition of the tricolouration given in the introduction to the random infinite limiting trees $T_1^*$, $T_2^*$ and $T_3^*$ defined in Section \ref{sec limit thm}. The initial definition applies only to finite trees since a covering of smallest size only makes sense in this context. Even though it seems not obvious to find a satisfactory definition of "smallest covering" for an infinite tree, it is still possible to describe a canonical way to tricolour the trees $T_1^*$, $T_2^*$ and $T_3^*$ using the properties found in Section \ref{sec prop tricolour}. 

Let $t$ be a subtree of $\mathcal U_\infty^\bullet$ (finite or not) such that for every $k \geq 0$ and for every child $v$ of $u_k$, distinct from $u_{k-1}$, $v$ has a finite number of descendants. In other words, for all $k\geq 0$, all the children of $u_k$, distinct from $u_{k-1}$, are roots of finite trees. Since those trees are finite, it makes sense to consider their tricolouration. A \textit{good} vertex of $t$ is a vertex $u_k$ with $k\geq 0$ such that, at least two of its children, distinct from $u_{k-1}$, are red in the tricolouration of the finite subtree they produce. If $t$ is finite, then, from Lemma \ref{lemme jointures}, a good vertex is a green vertex for the tricolouration of $t$. Notice that $T_1^*$, $T_2^*$ and $T_3^*$ satisfy the same hypothesis as $t$ almost surely.
\begin{itemize}
    \item We begin with the definition of the tricolouration of $T_1^*$. Almost surely, there exists an increasing sequence $(k_i)_i$ such that for all $i$, $u_{k_i}$ is good. For all $i$, all the vertices below $u_{k_i}$ ($u_{k_i}$ included) get the same colour in $T_1^*$ as their colour in the tricolouration of the subtree rooted at $u_{k_i}$ (which is a finite tree). Notice that, from Lemma \ref{lemme jointures}, $u_{k_i}$ gets necessarily the colour green. Lemma \ref{lemme jointure} ensures that this way of colouring is consistent when taking larger $i$.
    \item For $T_2^*$ we colour the unique vertex with infinite degree in green. Then, by cutting this vertex from $T_2^*$ we obtain a (infinite) forest of finite trees who gets their induced tricolouration. Notice that, almost surely, the vertex with infinite degree is good.
    \item Lastly, all the leaves of $T_3^*$ are coloured in red and the root $u_1$ is coloured in green.
\end{itemize}

We finish this section with the following lemma which explicitly gives the colour distribution of the vertex $u_0$ in $T_i^*$. This lemma will be useful when proving Theorem \ref{main thm}.

\begin{lemma}
\label{lemme calcul pic}
Let $p_i(c)$ be the probability that $u_0$ has colour $c$ in $T_i^*$.
\begin{enumerate}
\item In regime $i=1$ and $i=2$ with complete condensation we have
\begin{equation*}
 p_i(\text{green})=\frac{1-q+(1-2q)G'(1-q)}{1+G'(1-q)}, ~~~~~~~~ p_i(\text{orange})= \frac{2\,q\,G'(1-q)}{1+G'(1-q)},  
\end{equation*}
\begin{equation*}
    p_i(\text{red})= \frac{q}{1+G'(1-q)}.
\end{equation*}
\item In regime 3 with complete condensation we have that $p_3(\text{red})=1$.
\end{enumerate}
\end{lemma}

\begin{proof}
The case of regime 3 is obvious, let us focus on regime 1 and 2. Let $T$ be a BGW tree with reproduction law $\pi$. Denote by $q$ the probability that the root of $T$ is red. From Lemma \ref{lemme jointures} we deduce that
\begin{equation*}
q = \sum_{k\geq 0}\pi_k(1-q)^k = G(1-q).
\end{equation*}
Let $\widetilde T$ be the tree obtained from $T_i^*$ by cutting the edge between $u_0$ and $u_1$ and keeping the component containing $u_1$. Let $\widetilde q$ be the probability that $u_1$ is red in $\widetilde T$. Then, from Lemma \ref{lemme jointures} again,
\begin{equation*}
\widetilde q = \sum_{k\geq 1} k\pi_k(1-q)^{k-1} (1-\widetilde q) = (1-\widetilde q)G'(1-q).
\end{equation*}
Finally,
\begin{equation*}
p_i(\text{red})=\sum_{k\geq 0} \pi_k(1-q)^k (1-\widetilde q) = \frac{q}{1+G'(1-q)}. 
\end{equation*}
And
\begin{equation*}
p_i(\text{orange})=\sum_{k\geq 0} \pi_k(1-q)^k \widetilde q + \sum_{k\geq 1}k\pi_k(1-q)^{k-1}q(1-\widetilde q)=\frac{2qG'(1-q)}{1+G'(1-q)}.
\end{equation*}
Finally we deduce the value $p_i(\text{green})$ by the law of total probability.
\end{proof}

\section{Proof of Theorem \ref{main thm}}
\label{sec proof}

All this section is devoted to the proof of Theorem \ref{main thm}. We keep all the notation of Theorem \ref{main thm} and suppose that we are in regime $i=1$ or in regime $i\in \{2,3\}$ with complete condensation. Let $c$ be a colour in $\{\text{red, green, orange}\}$. Recall that $p_i(c)$ is the probability that the vertex $u_0$ has colour $c$ in the tree $T_i^*$ in regime $i$. The idea is to prove the convergence of the first two moments of $n_c(T_n)/n$, namely
\begin{equation*}
\frac{1}{n}\esp{n_c(T_n)} \xrightarrow[n \to \infty]{} p_i(c) ~~~~\text{and}~~~~ \frac{1}{n^2}\esp{n_c(T_n)^2} \xrightarrow[n \to \infty]{} p_i(c)^2.
\end{equation*}
Then, using Lemma \ref{lemme cvg lp}, we will conclude that $n_c(T_n)/n$ converges in $L^p$ towards $p_i(c)$ for all $p>0$. Actually, the convergence of the second moment won't be required in regime 3. Recall that the explicit computation of $p_i(c)$ can be found in Lemma \ref{lemme calcul pic}.
\begin{lemma}
\label{lemme cvg lp}
Let $(X_n)$ be a sequence of random variables with values in $[0,1]$, and $\alpha \in [0,1]$. Suppose that one of the following condition is satisfied.
\begin{itemize}
    \item The convergences $\esp{X_n} \to \alpha$ and $\esp{X_n^2} \to \alpha^2$ hold when $n\to \infty$.
    \item The convergence $\esp{X_n} \to \alpha$ holds when $n\to \infty$ and $\alpha=1$.
\end{itemize} 
Then for all $p>0$, $(X_n)$ converges towards $\alpha$ in $L^p$.
\end{lemma}
\begin{proof}[Proof of Lemma \ref{lemme cvg lp}]
First, we show that $(X_n)$ converges towards $\alpha$ in probability. Let $\varepsilon >0$. In the first case we use Markov's inequality which gives  
\begin{equation*}
\proba{|X_n-\alpha|\geq \varepsilon} \leq \frac{\esp{(X_n-\alpha)^2}}{\varepsilon^2} \xrightarrow[n \to \infty]{} 0.
\end{equation*}
In the second case we notice that
\begin{equation*}
\esp{X_n} \leq (1-\varepsilon)\proba{X_n\leq 1-\varepsilon}+\proba{X_n> 1-\varepsilon} = 1-\varepsilon \proba{X_n\leq 1-\varepsilon}.
\end{equation*}
Consequently $\lim_{n\to\infty}\proba{X_n\leq 1-\varepsilon} = 0$ and the convergence in probability is shown in both cases. Second, let $\ell$ be an accumulation point of the sequence $(\esp{|X_n-\alpha|^p})_n$ and $(n_k)_k$ be an extraction such that the convergence to $\ell$ occurs. From $(X_{n_k})_k$ we can extract a subsequence that converges almost surely to $\alpha$. From the dominated convergence theorem we deduce that $\ell = 0$.
\end{proof}
Let $v_n,v'_n$ be vertices chosen independently and uniformly in $T_n$. Notice that
\begin{equation*}
\frac{1}{n}\esp{n_c(T_n)} = \proba{v_n \text{ has colour }c\text{ in }T_n} ~~~~\text{and} 
\end{equation*}
\begin{equation*}
\frac{1}{n^2}\esp{n_c(T_n)^2} = \proba{v_n\text{ and }v'_n \text{ have colour }c\text{ in }T_n}.    
\end{equation*}
\paragraph*{Convergence of the first moment} First we prove the convergence of the first moment. Recall the notation of Section \ref{sec tricoloration arbre infini} when defining the tricolouration of the infinite trees $T_1^*$, $T_2^*$ and $T_3^*$. Denote by $k \geq 0$ the first positive integer such that $u_k$ is good ($k=1$ almost surely in regime 3). Let $\tau_i^*$ be the subtree of $\mathcal U_\infty^\bullet$ obtained from $T_i^*$ by cutting the edge between $u_k$ and $u_{k+1}$ and keeping the component containing $u_0$ ($\tau_3^*=T_3^*$ in regime 3). Note that, by construction, the tricolouration of $\tau_i^*$ is the restriction of its tricolouration in $T_i^*$ and that $u_k$ is green. The following definition introduces a useful order relation between trees.
\begin{definition}
\label{defi presque fringe subtree}
Let $T$ and $t$ be subtrees of $\mathcal U_\infty^\bullet$ such that $t$ is rooted at $u_j$ for some $j\geq 0$. Suppose that $u_j$ is also a vertex of $T$ and denote by $E_j$ the set of edges of $T$ adjacent to $u_j$. We write $t \preceq T$ if there exists a subset $e_j \subset E_j$ such that $t$ is the tree obtained from $T$ by cutting all the edges from $e_j$ and keeping the component containing $u_j$ (see Figure \ref{fig presque fringe subtree}).
\end{definition}

\begin{figure}
\begin{center}
\begin{tikzpicture}
    [vertexdot/.style = {draw,circle,fill,inner sep=1.5pt},scale=0.8]
    
    \node[vertexdot,label={[right]\footnotesize $u_0$}] (u0) at (0,0) {};
    \node[vertexdot,label={[right]\footnotesize $u_1$}] (u1) at (0,1) {};
    \node[vertexdot] (u1g) at (-0.7,1) {};
    \node[vertexdot] (u1d) at (0.7,1) {};
    \node[vertexdot] (u1dd) at (1.4,1) {};
    \node[vertexdot,label={[right]\footnotesize $u_2$}] (u2) at (0,2) {};
    \node[vertexdot] (u2g) at (-0.8,2.2) {};
    \node[vertexdot] (u2d) at (0.8,2.2) {};
    \node[vertexdot,label={[right]\footnotesize $u_3$}] (u3) at (0,3) {};
    
    \node (a) at (0,-0.4) {$a$};
    \node (b) at (-0.7,0.65) {$b$};
    \node (c) at (0.7,0.6) {$c$};
    \node (d) at (1.4,0.65) {$d$};
    
    \node (T) at (-0.7,3.2) {$T$};
    
    \draw (u0)--(u1)--(u2)--(u3);
    \draw (u1g)--(u2)--(u1d);
    \draw (u2)--(u1dd);
    \draw (u2g)--(u3)--(u2d);
    
    \draw (u0) to[out=-45,in=225,distance=1cm] (u0);
    \draw (u1g) to[out=-45,in=225,distance=1cm] (u1g);
    \draw (u1d) to[out=-45,in=225,distance=1cm] (u1d);
    \draw (u1dd) to[out=-45,in=225,distance=1cm] (u1dd);
    \draw (u2g) to[out=-45,in=225,distance=1cm] (u2g);
    \draw (u2d) to[out=-45,in=225,distance=1cm] (u2d);
\end{tikzpicture}
\begin{tikzpicture}
    [vertexdot/.style = {draw,circle,fill,inner sep=1.5pt},scale=0.8]
    
    \node[vertexdot] (u0) at (0,0) {};
    \node[vertexdot] (u1) at (0,1) {};
    \node[vertexdot] (u1g) at (-0.7,1) {};
    \node[vertexdot] (u1d) at (0.7,1) {};
    \node[vertexdot] (u2) at (0,2) {};
    
    \node (a) at (0,-0.4) {$a$};
    \node (b) at (-0.7,0.65) {$b$};
    \node (c) at (0.7,0.6) {$c$};
    
    \node (t1) at (-0.6,2.2) {$t_1$};
    
    \draw (u0)--(u1)--(u2);
    \draw (u1g)--(u2);
    \draw (u2)--(u1d);
    
    \draw (u0) to[out=-45,in=225,distance=1cm] (u0);
    \draw (u1g) to[out=-45,in=225,distance=1cm] (u1g);
    \draw (u1d) to[out=-45,in=225,distance=1cm] (u1d);
\end{tikzpicture}
\begin{tikzpicture}
    [vertexdot/.style = {draw,circle,fill,inner sep=1.5pt},scale=0.8]
    
    \node[vertexdot] (u0) at (0,0) {};
    \node[vertexdot] (u1) at (0,1) {};
    \node[vertexdot] (u1g) at (-0.7,1) {};
    \node[vertexdot] (u1dd) at (1.4,1) {};
    \node[vertexdot] (u2) at (0,2) {};
    
    \node (a) at (0,-0.4) {$a$};
    \node (b) at (-0.7,0.65) {$b$};
    \node (d) at (1.4,0.65) {$d$};
    
    \node (t2) at (-0.6,2.2) {$t_2$};
    
    \draw (u0)--(u1)--(u2);
    \draw (u1g)--(u2);
    \draw (u2)--(u1dd);
    
    \draw (u0) to[out=-45,in=225,distance=1cm] (u0);
    \draw (u1g) to[out=-45,in=225,distance=1cm] (u1g);
    \draw (u1dd) to[out=-45,in=225,distance=1cm] (u1dd);
\end{tikzpicture}
\begin{tikzpicture}
    [vertexdot/.style = {draw,circle,fill,inner sep=1.5pt},scale=0.8]
    
    \node[vertexdot] (u1g) at (-0.7,1.65) {};
    \node[vertexdot] (u2) at (0,2.65) {};
    \node[] (u0) at (0,0) {};
    
    \node (b) at (-0.7,1.3) {$b$};
    
    \node (t3) at (-0.6,2.85) {$t_3$};
    
    \draw (u1g)--(u2);
    
    \draw (u1g) to[out=-45,in=225,distance=1cm] (u1g);
\end{tikzpicture}
\hspace{1em}
\begin{tikzpicture}
    [vertexdot/.style = {draw,circle,fill,inner sep=1.5pt},scale=0.8]
    
    \node[vertexdot] (u0) at (0,0) {};
    \node[vertexdot] (u1) at (0,1) {};
    \node[vertexdot] (u1g) at (-0.7,1) {};
    \node[vertexdot] (u1d) at (0.7,1) {};
    \node[vertexdot] (u2) at (0,2) {};
    
    \node (a) at (0,-0.4) {$a$};
    \node (c) at (0.7,0.6) {$c$};
    
    \node (t4) at (-0.6,2.2) {$t_4$};
    
    \draw (u0)--(u1)--(u2);
    \draw (u1g)--(u2);
    \draw (u2)--(u1d);
    
    \draw (u0) to[out=-45,in=225,distance=1cm] (u0);
    \draw (u1d) to[out=-45,in=225,distance=1cm] (u1d);
\end{tikzpicture}
\caption{Illustration of Definitions \ref{defi sous arbre} and \ref{defi presque fringe subtree}. Only the tree $t_1$ satisfies $t_1 \preceq T$. The tree $t_3$ is not even a subtree of $\mathcal U^\bullet_\infty$ since it doesn't contain $u_0$. The tree $t_2$ is not a subtree of $\mathcal U^\bullet_\infty$ either since it has a hole between $u_1$ and the rightmost child of $u_2$. Finally $t_4$ is a subtree of $\mathcal U_\infty^\bullet$ but doesn't satisfy $t_4 \preceq T$ because the descendants of the left child of $u_2$ are missing.}
\label{fig presque fringe subtree}
\end{center}
\end{figure}
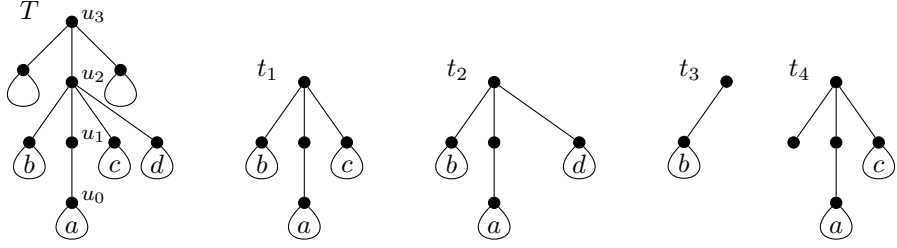

Let $\mathcal F$ be the set of rooted subtrees $t$ of $\mathcal U_\infty^\bullet$ such that the root $u_j \in t$ is the only good vertex of $t$. For $t \in \mathcal F$ such that the root $u_j$ of $t$ has finite degree, denote by $v_\ell(t)$ (resp. $v_r(t)$) the leftmost (resp. rightmost) child of the root $u_j$ of $t$. Let $\mathcal F_0$ be the set of elements $t \in \mathcal F$ such that: the root $u_j$ of $t$ has finite degree ; $u_j$ has exactly two red neighbors distinct from $u_{j-1}$ ; and ($v_\ell(t)=u_{j-1}$ or $v_\ell(t)$ is red) and ($v_r(t)=u_{j-1}$ or $v_r(t)$ is red). Notice that if $t_1$ and $t_2$ are distinct elements of $\mathcal F_0$, then we can't have $t_1 \preceq t_2$ nor $t_2 \preceq t_1$. Moreover for every $t_1 \in \mathcal F$ there exists a unique $t_2 \in \mathcal F_0$ such that $t_2 \preceq t_1$. In other words $\mathcal F_0$ is the set of equivalence classes for the equivalence relation $t_1\sim t_2$ iff $t_1 \preceq t_2$ or $t_2 \preceq t_1$. Notice that almost surely $\tau_i^* \in \mathcal F$.

Fix $\varepsilon >0$. Let $\mathcal T$ be a finite subset of $\mathcal F_0$ such that the event $\{\mathcal T \preceq \tau_i^*\}:=\{\exists t\in\mathcal T,\, t\preceq \tau_i^*\}$ happens with probability at least $1-\varepsilon$. Let $\mathcal T(c)$ be the set of trees $t\in \mathcal T$ such that $u_0$ has colour $c$ in $t$. Remember that we see $(T_n,v_n)$ as a subtree of $\mathcal U_\infty^\bullet$. For all $t\in \mathcal T$, define the event $A_n(t) := \{t \preceq (T_n,v_n)\}$. Using Theorem \ref{limit thm}, we have that for all $t \in \mathcal T$
\begin{equation}
\label{eq preuve premier moment}
\proba{A_n(t)} \xrightarrow[n \to \infty]{} \proba{t \preceq T_i^*}.    
\end{equation}
The properties of $\tau_i^*$ and $t\in\mathcal T$ imply that $t \preceq T_i^*$ if and only if $t \preceq \tau_i^*$. Thus
\begin{equation*}
\proba{t \preceq T_i^*} = \proba{t \preceq \tau_i^*}.
\end{equation*}
Denote by $E_n$ the event $\cup_{t\in\mathcal T}A_n(t)$. Notice that the event $E_n \cap \{v_n \text{ has colour }c \text{ in }T_n\}$ is equal to the event $\cup_{t\in\mathcal T(c)}A_n(t)$. It is a consequence of Lemma \ref{lemme jointures} and \ref{lemme jointure}. Notice also that for $t_1,t_2$ distinct trees of $\mathcal T$, $A_n(t_1)$ and $A_n(t_2)$ are disjoint for all $n$. Consequently,
\begin{equation*}
\proba{\{v_n \text{ has colour }c \text{ in }T_n\} \cap E_n} = \sum_{t \in \mathcal T(c)}\proba{A_n(t)} \xrightarrow[n \to \infty]{} \proba{\mathcal T(c) \preceq \tau_i^*}
\end{equation*}
Since $u_0$ has colour $c$ in $\tau_i^*$ if and only if $u_0$ has colour $c$ in $T_i^*$,
\begin{equation*}
\proba{\mathcal T(c)\preceq\tau_i^*} = \proba{u_0 \text{ has colour }c \text{ in }T_i^* \text{ and } \mathcal T\preceq\tau_i^*} \in [p_1(c)-\varepsilon,p_1(c)+\varepsilon].
\end{equation*}
Finally, using Theorem \ref{limit thm} again, one has
\begin{equation*}
\proba{E_n} \xrightarrow[n \to \infty]{} \proba{\mathcal T\preceq\tau_i^*} \geq 1-\varepsilon.
\end{equation*}
The convergence $\esp{n_c(T_n)/n} \to p_1(c)$ readily follows.

\paragraph*{Convergence of the second moment in regime 1 and 2} 
The next step is to show convergence for the second moment in regime $i=1$ or $i=2$ with complete condensation. Let $c'$ be another colour in $\{\text{red, green, orange}\}$. We will actually show that
\begin{equation*}
\frac{1}{n^2}\esp{n_c(T_n)n_{c'}(T_n)} \xrightarrow[n \to \infty]{} p_i(c)p_i(c').
\end{equation*}
We keep the notation of the previous part which shows the convergence of the first moment. For all $t\in\mathcal T$, let $A'_n(t) := \{t \preceq (T_n,v'_n)\}$ and $E'_n := \cup_{t\in\mathcal T}A'_n(t)$. We have,
\begin{equation}
\label{eq preuve second moment}
\proba{\{v_n \text{ has colour }c \text{ in }T_n\} \cap E_n \cap E'_n} = \sum_{t \in \mathcal T(c)}\sum_{t'\in\mathcal T(c')}\proba{A_n(t)\cap A'_n(t')}.    
\end{equation}
Fix $t,t'\in\mathcal T$. Recall that the trees $t,t'$ and $T_n$ are rooted plane trees, thus we can consider their so-called \L ukasiewicz walk. More precisely, let $T$ be a rooted plane tree with $n$ vertices and $w$ be a vertex of $T$. Denote by $\ell(w,T)$ the rank of $w$ in $T$ for the lexicographic order. Equivalently, $w$ is the $\ell(w,T)$-th vertex of $T$ explored by the depth first search starting from the root of $T$. Let $w_1,\dots,w_n$ be the vertices of $T$ ordered according to the lexicographic order (so $\ell(w_i,T)=i$ for all $i$). The \textit{\L ukasiewicz walk} associated with $T$ is the sequence $(s_k)_{1\leq k \leq n}$ such that $s_0=0$ and $s_{k}-s_{k-1}+1$ is the out-degree of $w_k$ for all $k \in \{1,\dots,n\}$. An important property of the \L ukasiewicz walk is that it uniquely encodes its tree, meaning that the tree $T$ can be retrieved from $(s_k)_{1\leq k \leq n}$. Let $(S_k^{(n)})_{0\leq k\leq n}$ be the \L ukasiewicz walk associated with the tree $T_n$. Let $X_1,\dots,X_n,\dots$ be i.i.d random variables such that $\proba{X_1=m}=\pi_{m+1}$ for all integer $m\geq -1$ and set $S_k := \sum_{i=1}^k X_i$ for all $k \geq 0$. It is well known that, the random walk $(S_k)_{0\leq k\leq n}$, starting at 0 and conditioned on reaching $-1$ for the first time at time $n$, has the same law as $(S_k^{(n)})_{0\leq k\leq n}$. Let $m:=|t|$ be the number of vertices of $t$ and $m':=|t'|$. Let $(s_k)_{0\leq k \leq m}$ and $(s'_k)_{0\leq k \leq m'}$ be, respectively, the \L ukasiewicz walks associated with $t$ and $t'$ and denote by $x_k:=s_{k+1}-s_k$ and $x'_k:=s'_{k+1}-s'_k$ the associated steps. Write $k_0:=\ell(u_0,t)$, $k'_0:=\ell(u_0,t')$, $i_n:=\ell(v_n,T_n)$ and $i'_n:=\ell(v'_n,T_n)$. The indices $i_n$ and $i'_n$ are independent random elements of $\{1,\dots,n\}$ with uniform distribution. The event $A_n(t)$ happens if and only if the \L ukasiewicz walk $(S_k^{(n)})_{0\leq k\leq n}$ coincides with $(s_k)_{0\leq k \leq m}$, up to a vertical shifting, on the interval $\llbracket i_n-k_0,i_n+m-k_0 \rrbracket$. The same goes for $A'_n(t')$. More precisely
\begin{equation*}
A_n(t) \cap A'_n(t') = \{X_{i+i_{n}-k_0}^{(n)}=x_i ~\forall i\in\llbracket 1,m\rrbracket \text{ and }X_{i+i'_{n}-k'_0}^{(n)}=x'_i ~\forall i\in \llbracket 1,m' \rrbracket\}.
\end{equation*}
Applying the reverse Vervaat transform, we can change the initial excursion type conditioning into a bridge type conditioning (see e.g.~\cite[Sec.\,6.1]{pitmanbook}). Namely
\begin{multline*}
\proba{A_n(t) \cap A'_n(t')} \\
= \proba{X_{i+i_{n}-k_0}=x_i ~\forall i\in\llbracket 1,m\rrbracket \text{ and }X_{i+i'_{n}-k'_0}=x'_i ~\forall i\in \llbracket 1,m' \rrbracket \,\big\vert\, X_n=-1}.
\end{multline*}
Denote by $D_n$ the event $\{\llbracket i_n-k_0+1,i_n-k_0+m \rrbracket \cap \llbracket i'_n-k'_0+1,i'_n-k'_0+m'\rrbracket = \emptyset\}$. This event has a probability tending to 1 and one can see that
\begin{multline*}
\proba{A_n(t)\cap A'_n(t') \cap D_n} \\
= \proba{X_i=x_i ~\forall i\in\llbracket 1,m\rrbracket \text{ and }X_{i+m}=x'_i ~\forall i\in \llbracket 1,m' \rrbracket \,\big\vert\, X_n=-1}\proba{D_n}.
\end{multline*}
According to \cite[Thm.~11.7]{janson} the steps $X_1,\dots,X_{m+m'}$ conditioned on $\{X_n=-1\}$ are asymptotically independent and the conditioning fades for large values of $n$, consequently
\begin{align*}
\lim_{n\to\infty}\proba{A_n(t) \cap A'_n(t')} &= \proba{X_i=x_i ~\forall i\in\llbracket 1,m\rrbracket}\proba{X_i=x'_i ~\forall i\in\llbracket 1,m'\rrbracket} \\
&= \lim_{n\to\infty} \proba{A_n(t)}\proba{A'_n(t')}.
\end{align*}
Finally, using (\ref{eq preuve premier moment}) and (\ref{eq preuve second moment}), we have that
\begin{equation*}
\lim_{n\to\infty}\proba{\{v_n \text{ has colour }c \text{ in }T_n\}\cap E_n \cap E'_n} = \proba{\mathcal T(c)\preceq\tau_i^*}\proba{\mathcal T(c)\preceq\tau_i^*}.
\end{equation*}
and the result follows.

\section*{Acknowledgements}
I am grateful to Igor Kortchemski for his careful reading of the manuscript and for telling me Frederic Chapoton’s suggestion to consider canonical tricolorations of random trees.

\bibliographystyle{alpha}
\bibliography{biblio}

\end{document}